\documentstyle[amscd,amssymb,verbatim,epsf]{amsart}
\begin{document}

\newtheorem{theorem}{Theorem}[section]
\newtheorem{algorithm}[theorem]{Algorithm}
\newtheorem{axiom}[theorem]{Axiom}
\newtheorem{case}[theorem]{Case}
\newtheorem{claim}[theorem]{Claim}
\newtheorem{conclusion}[theorem]{Conclusion}
\newtheorem{condition}[theorem]{Condition}
\newtheorem{conjecture}[theorem]{Conjecture}
\newtheorem{corollary}[theorem]{Corollary}
\newtheorem{criterion}[theorem]{Criterion}
\newtheorem{definition}[theorem]{Definition}
\newtheorem{example}[theorem]{Example}
\newtheorem{exercise}[theorem]{Exercise}
\newtheorem{lemma}[theorem]{Lemma}
\newtheorem{notation}[theorem]{Notation}
\newtheorem{problem}[theorem]{Problem}
\newtheorem{proposition}[theorem]{Proposition}
\newtheorem{remark}[theorem]{Remark}
\newtheorem{solution}[theorem]{Solution}
\newtheorem{proof}[theorem]{Proof}
\newtheorem{summary}[theorem]{Summary}
\renewcommand{\thenotation}{}

\errorcontextlines=0
\numberwithin{equation}{section}
\renewcommand{\rm}{\normalshape}%

\title[A new study on the strongly lacunary quasi Cauchyness]%
   {A new study on the strongly lacunary quasi Cauchyness}

\author[Huseyin Kaplan and Huseyin Cakalli]{Huseyin Kaplan* and Huseyin Cakalli** \\*Omer Halisdemir University, Department of Mathematics, Faculty of Science and Letters, Nigde, Turkey, Phone:(+90388)2254096 \\**Maltepe University, Graduate School of Science and Engineering, Istanbul, Turkey, Phone:(+90216)6261050 ext:2311,  fax:(+90216)6261113}

\address{Huseyin Kaplan\\
    Omer Halisdemir University, Department of Mathematics, Faculty of Science and Letters, Nigde, Turkey Phone:(+90388)2254096}
\email{hkaplan@@nigde.edu.tr}

\address{Huseyin Cakalli \\
          Maltepe University, Graduate School of Science and Engineering, Marmara E\u{g}\.{I}t\.{I}m K\"oy\"u, TR 34857, Maltepe, \.{I}stanbul-Turkey \; \; \; \; \; Phone:(+90216)6261050 ext:2311, \;  fax:(+90216)6261113}
\email{huseyincakalli@@maltepe.edu.tr; hcakalli@@gmail.com}

\keywords{summability,  strongly lacunary  convergence, quasi-Cauchy sequences, boundedness, continuity}
\subjclass[2010]{Primary: 40A05  Secondaries: 26A15, 40A30}
\date{\today}


\maketitle

\begin{abstract}

In this paper, the concept of an  $N_{\theta}^{2}$ quasi-Cauchy sequence is introduced. We proved interesting theorems related to  $N_{\theta}^{2}$-quasi-Cauchy sequences. A real valued function $f$ defined on a subset $A$ of $\mathbb{R}$, the set of real numbers, is  $N_{\theta}^{2}$ ward continuous on $A$ if it preserves  $N_{\theta}^{2}$  quasi-Cauchy sequences of points in $A$,  i.e.  $(f( \alpha_{k}))$ is an $N_{\theta}^{2}$  quasi-Cauchy sequence whenever $(\alpha_{k})$ is an $N_{\theta}^{2}$  quasi-Cauchy sequences of points in $A$, where a sequence $(\alpha_{k})$ is called $N_{\theta}^{2}$  quasi-Cauchy if $(\Delta^{2} \alpha_{k})$ is an $N_{\theta}$  quasi-Cauchy sequence where $\Delta^{2}\alpha_{k}=\alpha_{k+2}-2\alpha_{k+1}+\alpha_{k}$ for each positive integer $k$.
\end{abstract}

\maketitle

\section{Introduction}

The concept of continuity and any concept involving continuity play a very important role not only in pure mathematics but also in other branches of sciences involving mathematics especially in computer science, information theory, economics, and biological science.

Buck  (\cite{BuckSolutionofproblem4216AmerMathMonthly}) introduced Cesaro continuity in 1946. Thereafter, a number of authors Posner (\cite{PosnerSummabilitypreservingfunctions}), Antoni and Salat (\cite{AntoniandSalatOntheAcontinuityofrealfunctionsActaMathUnivComenian}), Spigel and Krupnik (\cite{SpigelandKrupnikOntheAcontinuityofrealfunctionsJAnal}) have studied $A$-continuity defined by a regular summability matrix $A$. Some authors,
Das and Sava\c{s} (\cite{DasandSavasOntheAcontinuityofrealfunctionsIstanbulUnivFenFakMatDerg}), Borsik and Salat (\cite{BorsikandSalatOnFcontinuityofrealfunctions}) have studied $A$-continuity for methods of almost convergence or for related methods. Connor and Grosse-Erdman (\cite{ConnorandGrosseErdmannSequentialdefinitionsofcontinuityforrealfunctionsRockyMountainJMath}) have given sequential definitions of continuity for real functions calling $G$-continuity instead of $A$-continuity by means of a sequential method, or a method of sequential convergence, and their results cover the earlier works related to $A$-continuity where a method of sequential convergence, or briefly a method, is a linear function $G$ defined on a linear subspace of all sequences of points in $\mathbb{R}$ denoted by $c_{G}$, into $\mathbb{R}$. A sequence $\boldsymbol{\alpha}=(\alpha_{k})$ is said to be $G$-convergent to $\ell$ if $\boldsymbol{\alpha}\in c_{G}$ ,then $G(\boldsymbol{\alpha})=\ell$. In particular, $\lim$ denotes the limit function $\lim \boldsymbol{\alpha}=\lim_{k} \alpha_{k}$ on the linear space $c$, where $c$ denotes the space of convergent sequences. On the other hand, \c{C}akall\i\; has introduced a generalization of compactness (\cite{CakalliSequentialdefinitionsofcompactness}), a generalization of connectedness (\cite{CakalliSequentialdefinitionsofconnectedness}), via a method of sequential convergence (see also \cite{CakalliandMucukconnectednessviaasequentialmethod} and \cite{MucukandSahanOnGsequentialcontinuity}).

In recent years, using the same idea, many kinds of continuities were introduced and investigated, not all but some of them we state in the following: slowly oscillating continuity (\cite{CakalliSlowlyoscillatingcontinuity}), \cite{Vallin}), quasi-slowly oscillating continuity (\cite{CanakandDik}), $\Delta$-quasi-slowly oscillating continuity  (\cite{CakalliandCanakandDik}, and  \cite{CakalliOnDeltaquasislowlyoscillatingsequences}), ward continuity (\cite{CakalliForwardcontinuity}, \cite{BurtonandColemanQuasiCauchySequences}),  $\delta$-ward continuity (\cite{CakalliDeltaquasiCauchysequences}),  $\delta^{2}$-ward continuity (\cite{BrahaandCakalliAnewtypecontinuityforrealfunctions}),  statistical ward continuity (\cite{CakalliStatisticalwardcontinuity}, \cite{CakalliStatisticalquasiCauchysequences}), lacunary statistical ward continuity  (\cite{CakalliandArasandSonmezLacunarystatisticalwardcontinuity}), $\rho$-statistically ward continuity (\cite{CakalliAvariationonstatisticalwardcontinuity}), $\lambda$-statistically ward continuity (\cite{CakalliandSonmezAraslambdastatisticallywardcontinuity}), and  $N_\theta-$ward continuity (\cite{CakalliNthetawardcontinuity}, \cite{CakalliandKaplanAstudyonNthetaquasiCauchysequences}, \cite{CakalliandKaplanAvariationonstronglylacunarywardcontinuityJournalofMathematicalAnalysis} \cite{KaplanandCakalliVariationsonstronglylacunaryquasiCauchysequencesAIP}, \cite{KaplanandCakalliVariationsonstronglacunaryquasiCauchysequencesJNonlinearSciAppl} . Investigation of some of these kinds of continuities lead some authors to find conditions on the domain of a function for some theorems related to uniform continuity of a real function in terms of sequences in the above manner (\cite[Theorem 8]{Vallin}, \cite[Theorem 7]{CakalliForwardcontinuity}, and \cite[Theorem 1]{BurtonandColemanQuasiCauchySequences}).

The notion of strongly lacunary or $N_\theta$ convergence was introduced, and studied by Freedman, Sember, and M. Raphael in \cite{FreedmanandSemberandRaphaelSomecesarotypesummabilityspaces} in the sense that a sequence $(\alpha_{k})$ of points in $\mathbb{R}$ is $N_\theta$ convergent to an $L\in{\mathbb{R}}$, which is denoted by $N_{\theta}-\lim \alpha_{k}=L$,  if $\lim_{r\rightarrow\infty}\frac{1}{h_{r}}\sum^{}_{k\in{I_{r}}}|\alpha_{k}-L|=0$, where $I_{r}=(k_{r-1},k_{r}]$, and $k_{0}\neq 0$, $h_{r}:k_{r}-k_{r-1}\rightarrow \infty$ as $r\rightarrow\infty$ and $\theta=(k_{r})$ is an increasing sequence of positive integers. Throughout this paper, it is assumed that $\liminf_{r}\frac{k_{r}}{k_{r-1}}>1$. The sums of the form $\sum^{k_{r}}_{k_{r-1}+1}|\alpha_{k}|$ frequently occur, and will often be written for convenience as $\sum^{}_{k\in{I_{r}}}|\alpha_{k}|$.

The purpose of this paper is to investigate the notion of $N_{\theta}^{2}$  ward continuity and  prove interesting theorems.

\maketitle

\section{$N_{\theta}^{2}$-ward continuity}
A function defined on a subset $A$ of $\mathbb{R}$  is called strongly lacunary continuous or $N_{\theta}$ continuous if it preserves $N_{\theta}$ convergent sequences of points in $A$, i.e. $(f(\alpha_{k}))$ is $N_{\theta}$ convergent whenever $(\alpha_{k})$ is an $N_{\theta}$ convergent sequence of points in $A$. A function defined on a subset $A$ of $\mathbb{R}$  is strongly lacunary continuous if and only if it is ordinary continuous.

A function defined on a subset $A$ of $\mathbb{R}$  is called strongly lacunary ward continuous or $N_{\theta}$-ward continuous if it preserves $N_{\theta}$-quasi-Cauchy sequences of points in $A$, i.e. $(f(\alpha_{k}))$ is $N_{\theta}$-quasi-Cauchy whenever $(\alpha_{k})$ is an $N_{\theta}$-quasi-Cauchy sequence of points in $A$ (see \cite{CakalliandKaplanAstudyonNthetaquasiCauchysequences}), where a sequence $(\alpha _{k})$ of points in $\mathbb{R}$ is called strongly lacunary quasi-Cauchy, or  $N_\theta$-quasi-Cauchy if $(\Delta \alpha _{k})$ is $N_\theta$-convergent to $0$ (\cite{CakalliNthetawardcontinuity, CakalliandKaplanAstudyonNthetaquasiCauchysequences}). A sequence $(\alpha _{k})$ of points in $\mathbb{R}$ is called strongly lacunary $\delta$-quasi-Cauchy, or $N_{\theta}$-$\delta$-quasi-Cauchy  if $(\Delta \alpha_{k})$ is an $N_\theta$ quasi-Cauchy sequence, i.e.
\[
\lim_{r\rightarrow\infty}\frac{1}{h_{r}}\sum^{}_{k\in{I_{r}}}|\Delta^{2}\alpha_{k}|=0
\]
where $\Delta^{2}\alpha_{k}=\alpha_{k+2}-2\alpha_{k+1}+\alpha_{k}$ for each positive integer $k$. $\Delta N_{\theta}$ and $\Delta^{2} N_{\theta}$   will denote the set of $N_{\theta}$-quasi-Cauchy sequences of points in $\mathbb{R}$, and the set of $N_{\theta}$-$\delta$-quasi-Cauchy sequences of points in $\mathbb{R}$, respectively.

\begin{definition}A sequence $(\alpha _{k})$ of points in $\mathbb{R}$ is called $N_{\theta}^{2}$-quasi-Cauchy, or strongly lacunary $\delta^{2}$-quasi-Cauchy  if $(\Delta^{2}(\alpha_{k}))$ is an $N_\theta$ quasi-Cauchy sequence, i.e.
\[
\lim_{r\rightarrow\infty}\frac{1}{h_{r}}\sum^{}_{k\in{I_{r}}}|\Delta^{3}\alpha_{k}|=0
\]
where $\Delta^{3}\alpha_{k}=\alpha_{k+3}-3\alpha_{k+2}+3\alpha_{k+1}-\alpha_{k}$ for each positive integer $k$.
\end{definition}

The sum of two $N_{\theta}^{2}$-quasi-Cauchy sequences is $N_{\theta}^{2}$-quasi-Cauchy, $(\lambda \alpha _{k})$ is $N_{\theta}^{2}$-quasi-Cauchy for every constant $\lambda\in{\mathbb{R}}$. Therefore the set of $N_{\theta}^{2}$-quasi-Cauchy sequences becomes a vector space that includes the vector space of convergent sequences, i.e. the vector space of all convergent sequences is a vector subspace of the vector space of all $N_{\theta}^{2}$-quasi-Cauchy sequences.
On the other hand, the vector space of all convergent sequences is also a vector subspace of the vector space of all $N_{\theta}$-quasi-Cauchy sequences, and the vector space of all $N_{\theta}$-quasi-Cauchy sequences is a vector subspace of the vector space of all $N_{\theta}^{2}$-quasi-Cauchy sequences.

Now we give some interesting examples that show importance of the interest.

\begin{example} Let $n$ be a positive integer. In a group of $n$ people, each person selects at random and simultaneously another person of the group. All of the selected persons are then removed from the group, leaving a random number $n_{1} < n$ of people which form a new group. The new group then repeats independently the selection and removal thus described, leaving $n_{2} < n_{1}$ persons, and so forth until either one person remains, or no persons remain. Denote by $\alpha_n$ the probability that, at the end of this iteration initiated with a group of $n$ persons, one person remains. Then the sequence $(\alpha_{1}, \alpha_{2}, · · ·, \alpha_{n},...)$  is an $N_{\theta}^{2}$  quasi-Cauchy sequence (see also \cite{WinklerMathematicalPuzzles}).
\end{example}

\begin{example} In a group of $k$ people, $k$ is a positive integer, each person selects independently and at random one of three subgroups to which to belong, resulting in three groups with random numbers $k_{1}$, $k_{2}$, $k_{3}$ of members; $k_{1} + k_{2} + k_{3} = k$. Each of the subgroups is then partitioned independently in the same manner to form three sub subgroups, and so forth. Subgroups having no members or having only one member are removed from the process. Denote by $\alpha_{k}$ the expected value of the number of iterations up to complete removal, starting initially with a group of $k$ people. Then the sequence $(\alpha_{1}, \frac{\alpha_{2}}{2}, \frac{\alpha_{3}}{3},...,\frac{\alpha_{n}}{n},...)$ is a bounded non-convergent $N_{\theta}^{2}$  quasi-Cauchy sequence (\cite{KeaneUnderstandingErgodicity}).
\end{example}

We note that any $N_{\theta}$-quasi Cauchy sequence is also $N_{\theta}-\delta^{2}$-quasi Cauchy, but the converse is not always true.
$\Delta^{3} N_{\theta}$ will denote the set of $N_{\theta}-\delta^{2}$-quasi-Cauchy sequences of points in $\mathbb{R}$.

Now we give the definition of $N_{\theta}-\delta^{2}$-ward compactness of a subset of $\mathbb{R}$.

\begin{definition}
A subset $A$ of $\mathbb{R}$ is called $N_{\theta}^{2}$  ward (or $N_{\theta}-\delta^{2}$-ward) compact if any sequence of points in $A$ has an $N_{\theta}-\delta^{2}$ quasi-Cauchy subsequence, i.e. if whenever $\boldsymbol{\alpha}=(\alpha_{k})$ is a sequence of points in $A$ there is a subsequence $\boldsymbol{\beta}=(\beta_{k})=(\alpha_{k_{k}})$ of $\boldsymbol\alpha$ with $N_{\theta}-\lim_{k\rightarrow \infty} \Delta^{3}\boldsymbol{\beta}_{k}=0$.
\end{definition}

Firstly, we note that any finite subset of $\mathbb{R}$ is $N_{\theta}-\delta^{2}$-ward compact, the union of finite number of $N_{\theta}-\delta^{2}$-ward compact subsets of $\mathbb{R}$ is $N_{\theta}-\delta^{2}$-ward compact, and the intersection of any family of $N_{\theta}-\delta^{2}$-ward compact subsets of $\mathbb{R}$ is $N_{\theta}-\delta^{2}$-ward compact. The sum of finite number of $N_{\theta}-\delta^{2}$-ward compact subsets of $\mathbb{R}$ is $N_{\theta}-\delta^{2}$-ward compact,  the product $\lambda A$ of an $N_{\theta}-\delta^{2}$-ward compact subset of $\mathbb{R}$ and a constant real number $\lambda$ is $N_{\theta}-\delta^{2}$-ward compact. Furthermore any subset of an $N_{\theta}-\delta^{2}$-ward compact set is $N_{\theta}-\delta^{2}$-ward compact and any bounded subset of $\mathbb{R}$ is $N_{\theta}-\delta^{2}$-ward compact. Any compact subset of $\mathbb{R}$ is also $N_{\theta}-\delta^{2}$-ward compact, and the converse is not always true. For example, any open bounded interval is $N_{\theta}-\delta^{2}$-ward compact, which is not compact. On the other hand, the set $\textbf{N}$ is not $N_{\theta}-\delta^{2}$-ward compact. We note that any slowly oscillating compact subset of $\mathbb{R}$ is $N_{\theta}-\delta^{2}$-ward compact (see also \cite{CakalliSlowlyoscillatingcontinuity}, and \cite{CakalliandCanakandDik} for the results on slowly oscillating compactness), and any quasi-slowly oscillating compact subset of $\mathbb{R}$ is $N_{\theta}-\delta^{2}$-ward compact where a subset $A$ of $\mathbb{R}$ is called quasi-slowly oscillating compact (see \cite{CanakandDik}) if whenever $\boldsymbol{\alpha} = (\alpha_n)$ is a sequence of points in $A$, there is a quasi-slowly oscillating subsequence $\boldsymbol{\beta} = (\beta_{n_{k}})$ of $\boldsymbol{\alpha}$.

We note that if a closed subset $E$ of $\mathbb{R}$ is $N_{\theta}$-ward compact, then it is also $N_{\theta}-\delta^{2}$-ward compact, and any sequence of points in $E$ has a $(P_{n} ,s)$-absolutely almost convergent subsequence (see \cite{CakalliandTaylanOnabsolutelyalmostconvergence}, \cite{OzarslanandYildizAnewstudyontheabsolutesummabilityfactorsofFourierseries}, \cite{YildizAnewtheoremonlocalpropertiesoffactoredFourierseries},  and \cite{YildizOnAbsoluteMatrixSummabilityFactorsofInfiniteSeriesandFourierSeries}).

Now we give the definition of $N_{\theta}^{2}$-ward continuity in the following.

\begin{definition}
A function defined on a subset $A$ of $\mathbb{R}$  is called $N_{\theta}^{2}$-ward continuous or $N_{\theta}-\delta^{2}$-ward continuous if it preserves $N_{\theta}-\delta^{2}$-quasi-Cauchy sequences of points in $A$, i.e. $(f(\alpha_{k}))$ is $N_{\theta}-\delta^{2}$-quasi-Cauchy whenever $(\alpha_{k})$ is an $N_{\theta}-\delta^{2}$-quasi-Cauchy sequence of points in $A$.
\end{definition}

We note that if  $f$ and $g$ are $N_{\theta}-\delta^{2}$-ward continuous functions on a subset $A$ of $\mathbb{R}$, then $f+g$ is $N_{\theta}-\delta^{2}$-ward continuous on $A$.
On the other hand, the product $\lambda f$ of an $N_{\theta}-\delta^{2}$-ward continuous function and a constant real number $\lambda$ is $N_{\theta}-\delta^{2}$-ward compact, i.e. $\lambda f$ is $N_{\theta}-\delta^{2}$-ward continuous whenever $f$ is an $N_{\theta}-\delta^{2}$-ward continuous function, and $\lambda$ is a constant real number, therefore the set of $N_{\theta}-\delta^{2}$-ward continuous functions becomes a vector space.
We note that the product of $N_{\theta}-\delta^{2}$-ward continuous functions need not be $N_{\theta}-\delta^{2}$-ward continuous. The function defined by $f(x)=x$ is clearly $N_{\theta}-\delta^{2}$-ward continuous whereas the product $f(x).f(x)=x^{2}$ is not $N_{\theta}-\delta^{2}$-ward continuous.



\begin{theorem} \label{TheoNthetadeltawardcontinuityimpliesNthetawardcontinuity} If $f$ is $N_{\theta}-\delta^{2}$-ward continuous on a subset $A$ of $\mathbb{R}$, then it is $N_{\theta}$-ward continuous on $A$.
\end{theorem}

\begin{proof}

Assume that $f$ is an $N_{\theta}-\delta^{2}$-ward continuous function on $A$. Let $(\alpha _{n})$ be any $N_{\theta}$-quasi-Cauchy sequence of points in $A$. Then the sequence $$(\alpha _{1}, \alpha _{1}, \alpha _{1}, \alpha _{2}, \alpha _{2}, \alpha _{2}, ..., \alpha _{n-1}, \alpha _{n-1}, \alpha _{n-1}, \alpha _{n}, \alpha _{n}, \alpha _{n},...)$$  is also $N_{\theta}$-quasi-Cauchy.  Hence it is $N_{\theta}-\delta^{2}$-quasi-Cauchy. As $f$ is $N_{\theta}-\delta^{2}$-ward continuous, the sequence  $$(f(\alpha _{1}),(f(\alpha _{1}), f(\alpha _{2}), f(\alpha _{2}),..., f(\alpha _{n-1}), f(\alpha _{n-1}), f(\alpha _{n}), f(\alpha _{n}),...)$$ is $N_{\theta}-\delta^{2}$-quasi-Cauchy. It follows from this that the sequence $(f(\alpha _{n}))$ is $N_{\theta}$-quasi-Cauchy. This completes the proof of the theorem.
\end{proof}

\begin{corollary} \label{CorNthetadeltawardcontinuityonanintervalimpliesordinarycontinuity} If $f$ is $N_{\theta}-\delta^{2}$-ward continuous on a subset $A$ of $\mathbb{R}$, then it is continuous on $A$.
\end{corollary}

\begin{proof}
The proof easily follows from Theorem \ref{TheoNthetadeltawardcontinuityimpliesNthetawardcontinuity} and \cite[Corollary 3]{CakalliandKaplanAvariationonstronglylacunarywardcontinuityJournalofMathematicalAnalysis}, so is omitted.

\end{proof}

\begin{theorem} \label{TheothecompositegofoffandgN_{theta}delta{2}wardcontinuous} The composite of two $N_{\theta}-\delta^{2}$-ward continuous functions is $N_{\theta}-\delta^{2}$-ward continuous, i.e. if $f$ and $g$ are $N_{\theta}-\delta^{2}$-ward continuous functions on $\mathbb{R}$, then the composite $gof$ of $f$ and $g$ is $N_{\theta}-\delta^{2}$-ward continuous.
\end{theorem}

\begin{proof}
Let $f$ and $g$ be $N_{\theta}-\delta^{2}$-ward continuous functions on $\mathbb{R}$, and $(\alpha _{n})$ be an $N_{\theta}-\delta^{2}$-quasi Cauchy sequence of points in $\mathbb{R}$. As $f$ is $N_{\theta}-\delta^{2}$-ward continuous, the transformed sequence $(f(\alpha _{n}))$ is an $N_{\theta}-\delta^{2}$-quasi Cauchy sequence. Since $g$ is $N_{\theta}-\delta^{2}$-ward continuous, the transformed sequence $g(f(\alpha _{n}))$ of the sequence $(f(\alpha _{n}))$ is an $N_{\theta}-\delta^{2}$-quasi Cauchy sequence.This completes the proof of the theorem.
\end{proof}

It should be noted that any $N_{\theta}-\delta^{2}$-ward continuous function is statistically continuous, $N_{\theta}$-continuous, $I$-sequentially continuous for any non-trivial admissible ideal $I$, and $G$-sequentially continuous for any regular subsequential sequential method $G$.

\begin{theorem} \label{NthetadeltawardcontinuousimageofNthetawardcompactsubsetisNthetadeltawardcompact} $N_{\theta}-\delta^{2}$-ward continuous image of any $N_{\theta}-\delta^{2}$-ward compact subset of $\mathbb{R}$  is $N_{\theta}-\delta^{2}$-ward compact.
\end{theorem}

\begin{proof}
Assume that $f$ is an $N_{\theta}-\delta^{2}$-ward continuous function on a subset $A$ of $\mathbb{R}$, and $A$ is an $N_{\theta}-\delta^{2}$-ward compact subset of  $\mathbb{R}$. Let $(\beta _{n})$ be any sequence of points in $f(A)$. Write $\beta _{n}=f(\alpha _{n})$ where $\alpha _{n} \in {A}$ for each positive integer $n$. $N_{\theta}-\delta^{2}$-ward compactness of $A$ implies that there is a subsequence $(\gamma _{k})=(\alpha _{n_{k}})$ of $(\alpha _{n})$ with $N_{\theta}-\lim_{k\rightarrow \infty} \Delta^{3} \gamma _{k}=0$. Write $(t_{k})=(f(\gamma_{k}))$. As $f$ is $N_{\theta}-\delta^{2}$-ward continuous, $(f(\gamma_{k}))$ is $N_{\theta}-\delta^{2}$-quasi-Cauchy. Thus we have obtained a subsequence $(t_{k})$ of the sequence $(f(\alpha _{n}))$ with $N_{\theta}-\lim_{k\rightarrow \infty} \Delta^{3} t_{k}=0$. Thus $f(A)$ is $N_{\theta}-\delta^{2}$-ward compact. This completes the proof of the theorem.
\end{proof}

\begin{corollary} \label{Nthetadeltawardcontinuousimageofcompactsubsetiscompact} $N_{\theta}-\delta^{2}$-ward continuous image of any $G$-sequentially connected subset of $\mathbb{R}$  is $G$-sequentially connected.
\end{corollary}

The proof follows from the preceding theorem, and \cite[Theorem 1]{CakalliSequentialdefinitionsofconnectedness}, so is omitted.

\begin{corollary} \label{Nthetawardcontinuousimageofboundedsubsetisbounded} $N_{\theta}-\delta^{2}$-ward continuous image of any bounded subset of $\mathbb{R}$  is bounded.
\end{corollary}

The proof follows from  \cite[Theorem 3.3]{CakalliNthetawardcontinuity}.

\begin{corollary} \label{NthetadeltawardcontinuousimageofGseqcompactsubsetisNthetawardcompactforregularsubsequentialmethodG} $N_{\theta}-\delta^{2}$-ward continuous image of a $G$-sequentially compact subset of $\mathbb{R}$ is $N_{\theta}-\delta^{2}$-ward compact for any regular subsequential method $G$.
\end{corollary}

As far as ideal continuity is considered, any $N_{\theta}-\delta^{2}$-ward continuous function is $I$-sequentially continuous for any admissible ideal $I$ (\cite{CakalliandHazarikaIdealquasiCauchysequences}). It follows from \cite[Theorem 1]{CakalliandKaplanAstudyonNthetaquasiCauchysequences} that $(f(\alpha_{k}))$ is $N_{\theta}-\delta^{2}$-quasi Cauchy whenever $(\alpha_{k})$ is a quasi-Cauchy sequence of points in a subset $A$ of $\mathbb{R}$ if $f$ is uniformly continuous on $A$.

It is a well known result that the uniform limit of a sequence of continuous functions is continuous. This is also true in case of $N_{\theta}-\delta^{2}$-ward  continuity, i.e. uniform limit of a sequence of $N_{\theta}-\delta^{2}$-ward continuous functions is $N_{\theta}-\delta^{2}$-ward continuous.

\begin{theorem} \label{Uniformlimitis} If $(f_{n})$ is a sequence of $N_{\theta}-\delta^{2}$-ward continuous functions on a subset $A$ of $\mathbb{R}$, and $(f_{n})$ is uniformly convergent to a function $f$, then $f$ is $N_{\theta}-\delta^{2}$-ward continuous on $A$.
\end{theorem}
\begin{proof}

Let $(\alpha _{k})$ be any  $N_{\theta}-\delta^{2}$-quasi-Cauchy sequence of points in $A$, and let $\varepsilon$ be any positive real number. By the uniform convergence of $(f_{n})$, there exists an $n_{1}\in {\textbf{N}}$ such that $|f(x)-f_{n}(x)|<\frac{\varepsilon}{9}$ for $n\geq {n_{1}}$ and for every $x \in{E}$. As $f_{n_{1}}$ is $N_{\theta}-\delta^{2}$-ward continuous on $A$, there exists an $n_{2}\in {\textbf{N}}$  such that for $r\geq {n_{2}}$
$$\frac{1}{h_{r}} \sum^{}_{k\in{I_{r}}} |f_{n_{1}}(\alpha_{k+3})-3f_{n_{1}}(\alpha_{k+2})+3f_{n_{1}}(\alpha_{k+1})-f_{n_{1}}(\alpha_{k})|<\frac{\varepsilon}{9}.$$
Now write $n_{0}=max\{n_{1},n_{2}\}$. Thus for $r\geq{n_{0}}$ we have
\\
$ \frac{1}{h_{r}}  \sum^{}_{k\in{I_{r}}} |f(\alpha _{k+3})-3f(\alpha _{k+2})+3f(\alpha _{k+1})-f(\alpha _{k})|
=\frac{1}{h_{r}}  \sum^{}_{k\in{I_{r}}} |f(\alpha _{k+3})-3f(\alpha _{k+2})+3f(\alpha _{k+1})-f(\alpha _{k})-[f_{n_{1}}(\alpha_{k+3})-3f_{n_{1}}(\alpha_{k+2})+3f_{n_{1}}(\alpha_{k+1})-f_{n_{1}}(\alpha_{k})]+[f_{n_{1}}(\alpha_{k+3})-3f_{n_{1}}(\alpha_{k+2})+3f_{n_{1}}(\alpha_{k+1})-f_{n_{1}}(\alpha_{k})]
\leq \frac{1}{h_{r}}  \sum^{}_{k\in{I_{r}}} |f(\alpha _{k+3})-f_{n_{1}}(\alpha_{k+3})|+ \frac{1}{h_{r}}  \sum^{}_{k\in{I_{r}}} 3|f(\alpha _{k+2})-f_{n_{1}}(\alpha_{k+2})|+ \frac{1}{h_{r}}  \sum^{}_{k\in{I_{r}}} 3|f(\alpha _{k+1})-f_{n_{1}}(\alpha_{k+1})| + \frac{1}{h_{r}}\sum^{}_{k\in{I_{r}}} |f(\alpha _{k})-f_{n_{1}}(\alpha_{k})|+ \frac{1}{h_{r}}\sum^{}_{k\in{I_{r}}} |f_{n_{1}}(\alpha_{k+3})-3f_{n_{1}}(\alpha_{k+2})+3f_{n_{1}}(\alpha_{k+1})-f_{n_{1}}(\alpha_{k})|$\\$<\frac{\varepsilon}{9}+\frac{\varepsilon}{9}+\frac{\varepsilon}{9}+\frac{\varepsilon}{9}<\varepsilon$
.
Hence
$$\lim_{r\rightarrow\infty}\frac{1}{h_{r}}  \sum^{}_{k\in{I_{r}}} |f(\alpha _{k+3})-3f(\alpha _{k+2})+3f_{n_{1}}(\alpha_{k+1})-f(\alpha _{k})|=0.$$
This completes the proof of the theorem.
\end{proof}

\begin{theorem} The set of $N_{\theta}^{2}$-ward continuous functions on a subset $A$ of $\mathbb{R}$ is a closed subset of the set of continuous functions on $A$. i.e. $\overline{\Delta^{3}N_{\theta}(A)} = \Delta^{3}N_{\theta}(A)$ where $~ \overline{\Delta^{3}N_{\theta}(A)}$ denotes the set of all cluster points of $\Delta^{3}N_{\theta}(A)$.
\end{theorem}

\begin{proof}  Let $f$ be an element in $\overline{\Delta^{3}N_{\theta}(A)}.$ Then there exists a sequence $(f_{n})$ of points in $\Delta^{3}N_{\theta}(A)$ such that $\lim_{n\rightarrow \infty} f_n =f.$
It follows from  Theorem \ref{Uniformlimitis} that $f\in{\Delta^{3}N_{\theta}(A)}$, which completes the proof of the theorem.
\end{proof}
\begin{corollary} The set of $N_{\theta}^{2}$-ward continuous functions on a subset $A$ of $\mathbb{R}$ is a complete subset of the set of continuous functions on $A$.
\end{corollary}

\section{Conclusion}
In this paper, the concept of an $N_{\theta}^{2}$-quasi-Cauchy sequence is introduced, and investigated. In this investigation, we proved interesting theorems related to $N_{\theta}^{2}$-ward continuity, and some other kinds of continuities. One may expect this investigation to be a useful tool in the field of analysis in modeling various problems occurring in many areas of science, dynamical systems, computer science, information theory, and biological science. For a further study, we suggest to investigate $N_{\theta}^{2}$-quasi-Cauchy sequences of fuzzy point, or soft points  (see \cite{CakalliandPratul}, \cite{ArasandSonmezandCakalliOnSoftMappings}, and \cite{EsiandAcikgozMehmetOnalmostlambdastatisticalconvergenceoffuzzynumbers} for the definitions and  related concepts), and we suggest to investigate $N_{\theta}^{2}$-quasi-Cauchy double sequences (see \cite{CakalliandSavasStatisticalconvergenceofdoublesequencesintopologicalgroups}, \cite{EsiAsymptoticallydoublelacunryequivalentsequencesdefinedbyOrliczfunctions}, \cite{PattersonandCakalliQuasiCauchydoublesequences}, and \cite{KocinacDoubleSequencesandSelections} for the concepts in double case). Another suggestion for another further study is to introduce and give investigations of $N_{\theta}^{2}$-quasi-Cauchy sequences of points in cone normed spaces (see \cite{CakalliandSonmezandGencOnanequivalenceoftopologicalvectorspacevaluedconemetricspacesandmetricspaces}, and \cite{SonmezandCakalliconenormedspacesandweightedmeans}, for basic concepts in topological vector space valued cone metric, and cone normed spaces). However due to the change in settings, the definitions and methods of proofs will not always be analogous to those of the present work.

\section{Acknowledgment}
We acknowledge that some of the results in this paper were presented at the International Conference on Recent Advances in Pure and Applied Mathematics (ICRAPAM 2017) May 11-15, 2017, Palm Wings Ephesus  Resort  Hotel, Kusadasi - Aydin, TURKEY.

\end{document}